\documentclass[12pt]{article}

\usepackage[margin=1in]{geometry}  % set the margins to 1in on all sides
\usepackage{graphicx}              % to include figures
\usepackage{amsmath}               % great math stuff
\usepackage{amsfonts}              % for blackboard bold, etc
\usepackage{amsthm}                % better theorem environments
\usepackage{amssymb}
\usepackage{bbm}

\newtheorem{thm}{Theorem}[]
\newtheorem{lem}[thm]{Lemma}

\newenvironment{definition}[1][Definition]{\begin{trivlist}
\item[\hskip \labelsep {\bfseries #1}]}{\end{trivlist}}

\newcommand{\RR}{\mathbb{R}}      % for Real numbers
\newcommand{\NN}{\mathbb{N}} 
\newcommand{\PP}{\mathbb{P}} 
\newcommand{\EE}{\mathbb{E}} 
\newcommand{\ind}{\mathbbm{1}}

\begin{document}

\nocite{*}

\title{Dean-Kawasaki Dynamics:\\
Ill-posedness vs.\ Triviality}

\author{Tobias Lehmann, Vitalii Konarovskyi, Max-K. von Renesse \\ 
\\
University of Leipzig}
   
\maketitle
 
\begin{abstract}
 We prove that the Dean-Kawasaki SPDE admits a solution only in integer 
parameter regimes, in which case the solution is given in terms of a 
system of non-interacting particles.
\end{abstract}

\begin{center}
1. I{\footnotesize INTRODUCTION}
\end{center}

In this paper we show a dichotomy of non-existence vs.\ triviality of solutions 
to a certain 
class of  nonlinear SPDE which arise e.g.\ in macroscopic fluctuation theory 
in physics. As a prototype one might consider the model  
\begin{equation}\label{Deanphys}
\partial_{t}\rho=T\Delta\rho+\nabla\cdot\left(\sqrt{T\rho}\dot{W}\right),
\end{equation} 
which is a particular instance of the more general class of 
formal Ginzburg-Landau stochastic phase field models 
 (cf.\cite{spohn1991large}) of the form  
\begin{equation*}
\partial_{t}\phi+\nabla\cdot \left(-L(\phi)\nabla\frac{\delta 
H}{\delta\phi}(\phi)+\sqrt{TL(\phi)}\dot{W} \right)=0,
\end{equation*} 
where $H$ is a Hamiltonian, $L$ is an Onsager coefficient and $\dot{W}$ is 
vector valued space-time white noise. The particular equation 
\eqref{Deanphys} 
was proposed 
independently in \cite{dean1996langevin} and \cite{Kawasaki1973} as mesoscopic 
description for interacting particles and is  referred to  as 
Dean-Kawasaki equation today. Since then, together with several variants, it  
has been an active research topic in various branches 
of non-equilibrium statistical mechanics over the last years
(c.f.\cite{delfau2016pattern,dirr,frusawa2000controversy,2014JPhA...47V5001J, PhysRevE.89.012150, 
marconi1999dynamic,velenich2008brownian}). Mathematically,  interest in 
equations like \eqref{Deanphys} comes from the fact that it appears to describe
 an 'intrinsic' random perturbation of the gradient flow for the 
entropy functional on Wasserstein space by a noise which is locally uniformly 
distributed in terms of dissipated energy, e.g.\ by a noise that is aligned 
with  Otto's formal Riemannian 
structure \cite{otto2001geometry} for optimal transportation. To see this, 
consider the rescaled  Hamiltonian ${
\mathcal{H}(\mu,\phi):=\lim_{\epsilon\to 0}\epsilon e^{-\frac{1}{\epsilon}\langle\mu,\phi\rangle}\mathcal{G}^{\epsilon}e^{\frac{1}{\epsilon}\langle\mu,\phi\rangle}}$,
with $\mathcal{G}^{\epsilon}$ being the generator associated to $\lbrace 
\mu_{\epsilon t}\rbrace_{t\ge 0}$.
Following \cite{DJEHICHE1999467}, one expects the short time asymptotics of $\mu$ to be governed by the large deviation rate function
\begin{equation*}
\mathcal{A}(\nu)=\int_{0}^{1}\sup_{\phi\in C^{\infty}_{0}(M)}\lbrace \langle\dot{\nu},\phi\rangle-\mathcal{H}(\nu,\phi)\rangle\rbrace dt,
\end{equation*}
for regular curves $\nu$ in the Wasserstein-space. But since 
$\mathcal{H}(\mu,\phi)=\frac{1}{2}\langle\mu,|\nabla\phi|^{2}\rangle$, we see 
that $\mathcal{A}$ is precisely the energy functional which determines the 
Wasserstein-metric by means of the Benamou-Brenier-formula (c.f. 
\cite{Benamou2000} and Appendix D in \cite{feng2006large}).\par 

In spite of significant continuing interest in  Dean-Kawasaki type models 
both  in physics and mathematics, 
rigorous results on existence and uniqueness exist so far only for appropriate 
regularisations of the original equation, see e.g. \cite{ 2018arXiv80201716C, fehrman2017well, 
mariani2010large}. On the other hand, in \cite{von2009entropic} Sturm and the 
third author succeeded in constructing a process having the Wasserstein distance 
as its intrinsic metric and which can formally be seen as a solution to the 
SPDE 
\begin{equation*}\label{vRS}
\partial_{t}\mu=\alpha\Delta\mu+\Xi(\mu)+\nabla\cdot\left(\sqrt{\mu}\dot{W}\right),
\end{equation*}
where $\Xi$ is some nonlinear operator. Particle approximations of these
dynamics as well as analytic and geometric properties of the corresponding 
entropic measure were investigated afterwards in 
\cite{andres2010particle,sturm2014monotone} and \cite{CHODOSH20124570, 
doring2009logarithmic}, respectively. Based on Arratia flows, a new candidate 
for a process with Wasserstein-short-time asymptotics, but with different drift 
component as in \cite{von2009entropic} was recently studied by the latter two 
authors in \cite{konarovskyi2015modified,2017arXiv70902839K} and subsequently, in \cite{marx2017new}.\par

\smallskip
The main contribution of this note asserts   
that some correction term $\Xi$  is in fact  necessary for 
the existence of 
(nontrivial)  solutions to these DK-type models.  More precisely, in  Theorem 
\ref{mainthm2} below  we find that  the (uncorrected) generalized  
Dean-Kawasaki equation
\begin{equation}\label{DE}
\begin{aligned}
\partial_{t}\mu&=\frac{\alpha}{2}\Delta\mu+\nabla\cdot\left(\sqrt{\mu}\dot{W}\right)\\
\mu_{|t=0}&=\mu_{0},
\end{aligned}\tag*{$(\text{DK})^{\alpha}_{\mu_{0}}$},
\end{equation}
corresponding to a Ginzburg-Landau model with $H = \alpha \mbox{Ent}$, $T=1$ 
and $L =$ identity operator, admits solutions only for a discrete spectrum of 
parameters $\alpha$ and atomic intial measures. Moreover, for these particular 
choices solutions are trivial, in being just 'measure-valued lifts' of the 
martingale-problem for $\frac{\alpha}{2}\Delta$.\par Finally,  given  the 
apparent similarity of \ref{DE} to the SPDE description 
of the Dawson-Watanbe ('Super Brownian Motion') process  
\[ \partial_{t}\mu = \beta \Delta \mu  + \sqrt \mu \dot{W},\]
admitting unique in law solutions for every $\beta >0$, our result is 
interesting also from an independent SPDE point of view. 
  
\pagebreak 
  
%------------------------------------------------------------------------------------------------------------ 
      
% \invisiblesection{proof}   
  
 \begin{center}
2. S{\footnotesize TATEMENT AND PROOF OF THE MAIN RESULT}
\end{center}
As for notation, given a Polish space $E$, we will write 
$\mathcal{M}_{1}(E)$ for the set of all probability measures on $E$ and for any 
function $f$ on $E$, we write as usual 
$\langle\mu,f\rangle=\int_{E}f(x)\mu(dx)$, whenever the integral is 
well-defined. By $C_{b}(E)$ we denote space of real-valued, bounded continuous 
functions on $E$.\par 
Let us briefly motivate, what we will refer to as a solution to the Dean-Kawasaki equation. Typically, one would call a time-continuous process $t\mapsto\mu_{t}$, which takes values in absolutely continuous measures on $\mathbb{R}^{d}$, a solution to \ref{DE}, provided that for all $t\in[0,T]$ and ${\phi\in\mathcal{S}(\RR^{d})}$ (Schwartz-space) it holds true that $\PP$-a.s
\begin{equation}\label{eqn: weak}
\langle\mu_{t},\phi\rangle=\langle\mu_{0},\phi\rangle+\frac{\alpha}{2}\int_{0}^{t}\langle\mu_{s},\Delta\phi\rangle ds-\sum_{i=1}^{d}\int_{0}^{t}\int_{\RR^{d}}\sqrt{\mu_{s}(x)}\partial_{i}\phi(x)W^{i}(dsdx),
\end{equation}
with $W^{i}$ being mutually independent space-time white noises (for instance in the sense of Walsh \cite{walsh1986introduction}).\par 

Of course, for such a process $\mu$ we knew that
\begin{equation*}
[0,T]\ni t\mapsto\langle\mu_{t},\phi\rangle-\langle\mu_{0},\phi\rangle-\frac{\alpha}{2}\int_{0}^{t}\langle\mu_{s},\Delta\phi\rangle ds
\end{equation*}
is martingale with quadratic variation
\begin{equation*}
\int_{0}^{t}\langle\mu_{s},|\nabla\phi|^{2}\rangle ds.
\end{equation*}
Rather then the weak formulation in (\ref{eqn: weak}), it is this martingale characterisation that we will study in the following, however in slightly more general setup.\par

Let $E$ be a Polish state-space and $(E, \pi, \Gamma)$ be the Markov-triple associated to some symmetric Markov Diffusion operator $L$ in the classical sense, as e.g. in \cite{bakry2013analysis}. In this setting, we know $L$ has the diffusion property, i.e.
\begin{equation*}
L\psi(f)=\psi'(f)Lf+\psi''(f)\Gamma f
\end{equation*}
for every $\psi\in C^{2}$ and $f\in\mathcal{D}(L)$. Here, $\Gamma$ is the is the carr\'e du champs operator, defined on some algebra $\mathcal{A}\subset C^{b}(E)$ which is dense in $L^{p}(\pi)$, by
\begin{equation*} 
\Gamma(f,g)=\frac{1}{2}(L(fg)-fLg-gLf)
\end{equation*}
for $(f,g)\in\mathcal{A}\times\mathcal{A}$ and $\Gamma f=\Gamma(f,f)$.
Additionally, we impose henceforth the following regularity hypothesis: the Markov-semigroup $P_{t}$ belonging to $L$, satisfies a gradient bound, i.e. for some $\rho\in\RR$, we have
\begin{equation}\label{eqn: gradbound}
\Gamma P_{t}f\le e^{-2\rho t}P_{t}\Gamma f.
\end{equation}
This requirement is certainly fulfilled for the Brownian semigroup on Euclidean space, but also for instance when $L=\Delta_{\mathfrak{g}}$ is the Laplace-Beltrami operator on a Riemannian manifold $(E=M,\mathfrak{g})$ with Ricci curvature bounded from below.

\par 

\begin{definition}
Let $\alpha>0$, $\mu_{0}\in\mathcal{M}^{1}(E)$ and $(\Omega,\mathcal{F},(\mathcal{F}_{t})_{t\ge0},\PP)$ be some probability base. We say that a $\PP$-almost surely continuous $\mathcal{M}^{1}(E)$-valued process $\mu$ is a solution to the martingale problem $(\text{MP})^{\alpha}_{\mu_{0}}$ iff
\begin{enumerate}
\item for all $\phi\in\mathcal{D}(L)$
\begin{equation}\label{mart}
M_{t}(\phi):=\langle\mu_{t},\phi\rangle-\langle\mu_{0},\phi\rangle-\frac{\alpha}{2}\int_{0}^{t}\langle\mu_{s},L\phi\rangle ds,\quad t\in[0,T],
\end{equation} 
is an $\mathcal{F}_{t}$-adapted martingale, 
\item whose quadratic variation is given by
\begin{equation*}
\langle M(\phi)\rangle_{t}=\int_{0}^{t}\langle\mu_{s},\Gamma\phi\rangle ds.
\end{equation*}
\end{enumerate}
\end{definition}  

With this notation, our main results reads as follows. 

\begin{thm}\label{mainthm2}
Solutions to $(\operatorname{MP})^{\alpha}_{\mu_{0}}$ exist, if and only if $\alpha=n\in\NN$ and 
\begin{equation}
\mu_{0}=\frac{1}{n}\sum_{i=1}^{n}\delta_{x_{i}},
\end{equation}
with $x_{1},\dots,x_{n}\in E$. In case of existence, the solution is, uniquely in law, given by the empirical measure
\begin{equation}\label{eqn: empmeas}
\mu_{t}=\frac{1}{n}\sum_{i=1}^{n}\delta_{X^{i}_{nt}},
\end{equation}
where the $X^{i}$ are $n$ independent diffusion processes, each with generator $\frac{1}{2}L$ and starting point $x_{i}$.
\end{thm} 

Certainly, the Dean-Kawasaki dynamic fits in as the special case of taking $L=\Delta$ on $E=\RR$.\par 

Theorem \ref{mainthm2} will be proven in several steps. We start by the almost trivial observation that the empirical measures (\ref{eqn: empmeas}) provide solutions to $(\operatorname{MP})^{n}_{\mu_{0}}$. Consider first the case $n=1$. Plugging $\mu_{t}=\delta_{X_{t}}$ into the martingale problem, immediately yields that
\begin{equation*}
M_{t}(\phi)=\phi(X_{t})-\phi(X_{0})-\frac{1}{2}\int_{0}^{t}L\phi(X_{s})ds
\end{equation*}
is a martingale, plainly since $X$ is the solution of the martingale problem for $\frac{1}{2}L$. Also necessarily, the quadratic variation of $M$ satisfies
\begin{equation*}
\langle M(\phi)\rangle_{t}=\int_{0}^{t}\Gamma\phi(X_{s})ds=\int_{0}^{t}\langle\mu_{s},\Gamma\phi\rangle ds.
\end{equation*}
For the general case, denote by $M^{i}$ the martingale associated to $X^{i}$. Then
\begin{equation*}
M_{t}(\phi)=\frac{1}{n}\sum_{i=1}^{n}\left(\phi(X^{i}_{nt})-\phi(X^{i}_{0})-\frac{n}{2}\int_{0}^{t}L\phi(X^{i}_{ns})ds\right)=\frac{1}{n}\sum_{i=1}^{n}M^{i}_{nt}(\phi)
\end{equation*}
is a $\mathcal{F}_{nt}$-martingale with quadratic variation
\begin{equation*}
\langle M(\phi)\rangle_{t}=\frac{1}{n^{2}}\sum_{i=1}^{n}\langle M^{i}(\phi)\rangle_{nt}=\frac{1}{n}\sum_{i=1}^{n}\int_{0}^{t}\Gamma\phi(X_{ns})ds=\int_{0}^{t}\langle\mu_{s},\Gamma\phi\rangle ds,
\end{equation*} 
as needed.\par
Our next aim is to show that these solutions are unique in law. In fact, the statement that we prove is much stronger, namely that any solution, provided its existence, must be unique in law. The proof adopts the argument which is used in order proof weak uniqueness for super-Brownian motion, by Laplace-duality to some reaction-diffusion equation (c.f.\cite{fitzsimmons1988construction, mytnik1998weak, perkins2002part}).\par 

Before we present the duality statement in our context, we provide some preliminary considerations on viscous Hamilton-Jacobi equations. That is, we regard for some initial datum $f\in C_{b}(E)$, the PDE
\begin{equation}\label{eqn: vHJ}
\begin{aligned} 
\partial_{t} v&=\frac{\alpha}{2}Lv-\frac{1}{2}\Gamma v,\\
v_{|t=0}&=f, 
\end{aligned}\tag*{$(\text{vHJ})_{f}$}
\end{equation} 
where $L$ and $\Gamma$ are as before generator and carr\'e du champs operator of some symmetric Markov diffusion. Of course, the unique solution is just the classical Cole-Hopf-solution given by
\begin{equation*}
V_{t}f(x)=v(t,x)=-\alpha\ln (P_{t}e^{-\frac{1}{\alpha} f})(x),
\end{equation*}
where $P_{t}$ is now the heat semi-group associated to $\frac{\alpha}{2}L$. Indeed, plugging in the definition of $V_{t}f$ into the PDE, we see
\begin{equation*}
\partial_{t} V_{t}f=-\alpha ({P_{t}e^{-\frac{1}{\alpha} f}})^{-1}LP_{t}e^{-\alpha f}=-\frac{\alpha^{2}}{2} e^{\frac{1}{\alpha} V_{t}f}Le^{-\frac{1}{\alpha} V_{t}f}=\frac{\alpha}{2} LV_{t}f-\frac{1}{2}\Gamma V_{t}f.
\end{equation*}
Moreover, one can easily check that this solution satisfies maximum/minimum-principles, i.e.
\begin{equation*}
\inf_{E}f\le\inf_{E}V_{t}f\qquad\text{ and }\qquad\sup_{E}V_{t}f\le\sup_{E
}f
\end{equation*}
for $f\in C_{b}(E)$. Also, by the gradient bound (\ref{eqn: gradbound}), there is $\rho\in\RR$ such that for all $f\in\mathcal{A}$, it follows that
\begin{equation}\label{eqn: gradest}
\Gamma V_{t}f \le \alpha^{2}(P_{t}e^{-\frac{1}{\alpha}f})^{-2}e^{-2\rho t}P_{t}\Gamma e^{-\frac{1}{\alpha}f}\le  1\vee e^{2|\rho|T}e^{\frac{2}{\alpha}\operatorname{diam}f(E)}\|\Gamma f\|_{\infty}.
\end{equation} 
  
We can now formulate our duality result, which will be crucial not only for the uniqueness, but also for the discussion of non-existence later on.   

\begin{thm}\label{thm: duality}
Take $\alpha>0$. For $\mu_{0}\in\mathcal{M}_{1}(E)$ and $f\in C_{b}(E)$, let $\mu$ be a solution to $(\operatorname{MP})^{\alpha}_{\mu_{0}}$ and $V_{t}f$ a solution to $(\operatorname{vHJ})_{f}$.
Then
\begin{equation}\label{eqn: duality}
\EE^{\mu_{0}}e^{-\langle\mu_{t},f\rangle}=e^{-\langle\mu_{0},V_{t}f\rangle}.
\end{equation}
\end{thm}

\begin{proof}
Assume for now $f\in\mathcal{A}$ and let $v(t,x)=V_{t}f(x)$ be the Cole-Hopf solution to $(\operatorname{vHJ})_{f}$. For $0\le s\le t\le T$, by It\={o}'s formula
\begin{align*}
de^{-\langle\mu_{s},v(t-s)\rangle}&=-e^{-\langle\mu_{s},v(t-s)\rangle}\left(\langle\mu_{s},(\partial_{s}+\frac{\alpha}{2}L)v(t-s)\rangle ds+dM_{s}(v(t-s))-\frac{1}{2}d\langle M(v(t-\cdot)\rangle_{s}\right)\\
&=e^{-\langle\mu_{s},v(t-s)\rangle}\left(\langle\mu_{s},(\partial_{s}v)(t-s)-\frac{\alpha}{2}L v(t-s)+\frac{1}{2}\Gamma v(t-s)\rangle ds+dM_{s}(v(t-s))\right)\\
&=e^{-\langle\mu_{s},v(t-s)\rangle}dM_{s}(v(t-s)).
\end{align*}
Hence for $t\in[0,T]$ and $s\le t$ the map $s\mapsto e^{-\langle\mu_{s},v(t-s)\rangle}$ is a local martingale. Since by (\ref{eqn: gradest})
\begin{equation*}
\EE\int_{0}^{T}e^{-2\langle\mu_{s},v(t-s)\rangle} \langle \mu_{s},\Gamma v(t-s)\rangle ds\le CTe^{(2+\frac{2}{\alpha})\operatorname{diam}f(E)}\|\Gamma f\|_{\infty}<\infty,
\end{equation*}
it is in fact a martingale. Therefore, upon choosing $s=t$ we find
\begin{equation*}
\EE^{\mu_{0}}e^{-\langle\mu_{t},f\rangle}=e^{-\langle\mu_{0},V_{t}f\rangle}.
\end{equation*}
By density of $\mathcal{A}$ in $C_{b}(E)$ the previous equation also holds for $f\in C_{b}(E)$, which yields the claim.
\end{proof}
Observe in particular, that the duality in Theorem \ref{thm: duality} determines the Laplace-transform of $\mu_{t}$ and by the same argument as in Theorem 11 of \cite{ELKAROUI1991239}, also the finite dimensional distributions of $\mu$ uniquely. Hence, we obtain uniqueness (in law) for solutions to the Dean-Kawasaki equation.\par
In order to develop our non-existence statement, we insert a short intermezzo on generating functions.\par 
Whereas of course, for the probability-generating function
\begin{equation}
g(s):=\mathbb{E}s^{X}=\sum_{k=0}^{\infty}p_{k}s^{k}
\end{equation}
of a some discrete random variable $X$ with values in $\NN^{0}=\{0,1,2,\dots\}$, one knows that ${p_{k}=\\P(X=k)}$, we are interested in the opposite direction and establish

\begin{lem}\label{lemma_expansion}
	Let $X$ be a non negative random variable, such that for each $n\in\NN^{0}$ 
	\begin{equation}\label{f_taylor_exp}
	g(s)=\EE s^{X}=\sum_{k=0}^ns^kp_k+o(s^n),\quad \mbox{as}\ \ s\to 0+,
	\end{equation}
	for some sequence $\{p_k,\  k\in\NN^{0}\}$. Then $X\in\NN^{0}$ a.s., $p_k\geq 0$ and 
	$$
	\PP\lbrace X=k\rbrace=p_k
	$$
	for each $k\in\NN$.
\end{lem}

\begin{proof}
	We will prove the statement by induction. Due to (\ref{f_taylor_exp}),
	$$
	\EE s^{X}\to p_0 \quad \mbox{as}\ \  s\to 0+.
	$$ 
	On the other hand, since $s^{X}\rightarrow \ind_{\{X=0\}}$ a.s. as $s\to 0+$, we infer by dominated convergence 
	$$
	\EE s^{X}\to\PP\{X=0\} \quad \mbox{as}\ \  s\to 0+.
	$$ 
	Hence, $p_0$ is non negative and equals $\PP\{X=0\}$.
	
	For the induction step, let us assume that $\PP\{X=k\}=p_k$ for $k=1,\dots,n-1$ and ${X\in \lbrace0,1,\dots,n-1\rbrace\cup (n-1,\infty)}$ a.s. Again, by~\eqref{f_taylor_exp},
	$$
	\frac{1}{s^n}\left(\EE s^{X}-\sum_{k=0}^{n-1}p_k s^k\right)\to p_n\quad\mbox{as}\ \ s\to 0+.
	$$
	Using the induction assumption, we can write
	\begin{align*}
	\frac{1}{s^n}\left(\EE s^{X}-\sum_{k=0}^{n-1}p_k s^k\right)&=\frac{1}{s^n}\EE\left[ s^{X}-\sum_{k=0}^{n-1}\ind_{\{X=k\}}s^{X}\right]=\frac{1}{s^n}\EE\left[ \ind_{\{\xi>n-1\}}s^{X}\right]\\
	&=\EE\left[\ind_{\{X\in(n-1,n)\}}s^{X-n}\right]+\EE\left[\ind_{\{X\geq n\}}s^{X-n}\right].
	\end{align*}
	Again, by the dominated convergence theorem,
	$$
	\EE\left[\ind_{\{X\geq n\}}s^{X-n}\right]\to \PP\{X=n\}\quad\mbox{as}\ \ s\to 0+.
	$$
	Thus, $\EE\left[s^{X-n}\ind_{\{X\in(n-1,n)\}}\right]$ is bounded for $s\in(0,1]$, which implies
	$$
	\PP\{X\in(n-1,n)\}=0.
	$$ 
	This finishes the proof of the lemma.
\end{proof}  
  
We can now return to our main question and prove, that the trivial solutions we found for $\alpha\in\mathbb{N}$ and atomic $\mu_{0}$ must in fact be the only possible ones.
\begin{proof}[Proof of Theorem \ref{mainthm2}]
	Take $\alpha>0$, $\mu_{0}\in\mathcal{M}_{1}(E)$ and a solution $\mu$ to $(\text{MP})^{\alpha}_{\mu_{0}}$. Also, for $A\in\mathcal{B}(E)$ and fixed $t\in[0,T]$, let us abbreviate $h(x)=P_{t}\ind_{A}(x)$. Note that for bounded $A$, we can find $\delta>0$ with $0\le h(x)\leq 1-\delta$ for all $x\in E$.
	
	For such an $A$ and fixed $t>0$, let us consider the generating function $g$ of the real-valued random variable ${X=\alpha\mu_{t}(A)}$. 
The Laplace-duality of Theorem \ref{thm: duality} yields
\begin{equation}
\EE e^{-r\alpha\mu_{t}(A)}=\EE e^{-\langle\mu_{t},r\alpha\ind_{A}\rangle}=e^{\langle\mu_{0},\ln(P_{t}e^{-r\ind_{A}})\rangle}.
\end{equation}
But since $P_{t}e^{-r\ind_{A}}=1+(e^{-r}-1)h$ and setting $s=e^{-r}$, the previous display reads
\begin{equation}
g(s)=e^{\alpha\langle\mu_{0},\ln(1+(s-1)h)\rangle}=e^{\alpha \left\langle\mu_0,\ln (1-h)+\ln\left(1+\frac{h}{1-h}s\right)\right\rangle}.
\end{equation}	
	
	By the boundedness of $h$, the function $g$ is well-defined on $(-\delta,\infty)$. Moreover, it is infinitely differentiable. Thus, for each $n\in\NN^{0}$  
	$$
	g(s)=\sum_{k=0}^np_ks^k+o(s^n)\quad\mbox{on}\ \ (-\delta,\delta),
	$$
	by Taylor's theorem. Consequently, by Lemma~\ref{lemma_expansion} we know that $\alpha\mu_t(A)\in\NN^{0}$ a.s. So, we have proved that for each bounded $A\in\mathcal{B}(\RR^d)$ and $t>0$
	$$
	\mu_t(A)\in\left\{0,\frac{1}{\alpha},\ldots,\frac{\lfloor \alpha\rfloor}{\alpha}\right\}\quad \mbox{a.s.}
	$$
	Here, we also used the fact that $\mu_t$ is a probability measure a.s.  Next, making $A\uparrow E$, we obtain that 
	$$
	1=\mu_t(E)\leq\frac{\lfloor \alpha\rfloor}{\alpha}\leq 1\quad\mbox{a.s.}
	$$
	This implies that $\alpha\in\NN$. In order to make a conclusion about $\mu_0$, we take a bounded set $A$ with $\mu_0$-zero boundary and use the continuity of the process $\mu$. So, we obtain
	$$
	\mu_0(A)\in \left\{0,\frac{1}{\alpha},\ldots,\frac{\lfloor \alpha\rfloor}{\alpha}\right\}.
	$$  
	Hence, there exist $x_k$, $k\in\{1,\ldots,\alpha\}$ such that 
	$$
	\mu_0=\frac{1}{\alpha}\sum_{k=1}^{\alpha}\delta_{x_k}.
	$$
\end{proof}

\noindent\textbf{Acknowledgements} \quad The first author gratefully acknowledges funding by the European Social Fund through the project 100234741. The research of the second author was supported by Alexander von Humboldt Foundation.

\bibliographystyle{plain}

\bibliography{exist}

\end{document}